[1994/06/01]

\documentclass[12pt,reqno,intlimits,english,centertags]{amsart}

\usepackage[T1]{fontenc}

\usepackage{ifthen,calc}
\usepackage{babel}
\usepackage{xspace}
\usepackage{amsmath}

\usepackage{lmodern}
\usepackage{amsthm}

\numberwithin{equation}{section}
\theoremstyle{plain}
\newtheorem{theorem}{Theorem}[section]
\newtheorem{lemma}[theorem]{Lemma}

\theoremstyle{remark}
\newtheorem{remark}[theorem]{Remark}
\theoremstyle{definition}
\newtheorem{definition}[theorem]{Definition}

\DeclareMathOperator{\Div}{div}
\newcommand{\abs}[1]{\lvert#1\rvert}

\newcommand{\norm}[1]{\lVert#1\rVert}

\newcommand{\norma}[2]{\norm{#1}_{#2}}

\newcommand{\SpDim}{N}
\newcommand{\numberspacefont}{\boldsymbol}
\newcommand{\R}{\numberspacefont{R}}
\newcommand{\RN}{\R^{\SpDim}}

\newcommand{\Pposbase}[3][*]{\ifthenelse{\equal{#1}{*}}%
{(#2)_{#3}}{\left(#2\right)_{#3}}}
\newcommand{\ppos}[1]{\Pposbase[*]{#1}{+}}

\newcommand{\di}{\,\textup{\textmd{d}}}
\newcommand{\eps}{\varepsilon}

\newcommand{\pder}[2]{\frac{\partial #1}{\partial #2}}

\newcommand{\emf}{G}

\newcommand{\grad}{\operatorname{\nabla}}
\DeclareMathOperator{\supp}{supp}
\newcommand{\der}[2]{\frac{\di #1}{\di #2}}

\newcommand{\Ka}{\mathcal{K}}

\newcommand{\smo}{\textup{\textmd{o}}}
\newcommand{\ew}{f}
\newcommand{\exw}{g}

\newcommand{\msw}{\mu_\ew}
\newcommand{\CH}{\varGamma}
\newcommand{\Har}{H_r}
\newcommand{\Kas}{K_s}

\newcommand{\fkf}{\varLambda}

\newcommand{\excf}{G}

\begin{document}
\title
[Equations with weights]{The Cauchy problem for doubly degenerate
  parabolic equations with weights}%
\author{Daniele Andreucci}
\address{Department of Basic and Applied Sciences for Engineering\\Sapienza University of Rome\\via A. Scarpa 16 00161 Rome, Italy}
\email{daniele.andreucci@uniroma1.it}
\thanks{The first author is member of the Gruppo Nazionale
  per la Fisica Matematica (GNFM) of the Istituto Nazionale di Alta Matematica
  (INdAM)}
\author{Anatoli F. Tedeev}
\address{Southern Mathematical Institute of VSC RAS\\53 Vatutina St. Vladikavkaz 362025, Russian Federation}
\email{a\_tedeev@yahoo.com}
\thanks{The second author  was supported by North-Caucasus Centre of Mathematical Research of the Vladikavkaz Scientific Centre of the Russian Academy of Sciences, agreement 075-02-2023-914.}

\date{\today}

\begin{abstract}
  We consider the Cauchy problem in the Euclidean space for a
  doubly degenerate parabolic equation with a space-dependent exponential weight, roughly speaking of the type of the exponential of a power of the distance from the origin. We assume here the solutions of the Cauchy problem to be globally integrable in space (in the appropriate weighted sense) and non-negative. Under suitable assumptions, we prove for the solutions sup estimates, i.e., the decay rate at infinity, the property of finite speed of propagation and
  support estimates. All our estimates are given explicitly in terms of the weight appearing in the equation.
\end{abstract}

\maketitle

\section{Introduction}
\label{s:intro}

Consider for $S_{T}=\RN\times (0,T)$, $0<T\le +\infty$, the Cauchy problem for a doubly degenerate weighted parabolic
equation
\begin{alignat}{2}
  \label{eq:pde}
  \ew(x)
  \pder{u}{t}
  -
  \Div\big(
  \ew(x)
  u^{m-1}
  \abs{\grad u}^{p-2}
  \grad u
  \big)
  &=
    0
    \,,
  &\qquad&
           \text{in $S_{T}$,}
  \\
  \label{eq:initd}
  u(x,0)
  &=
    u_{0}(x)
    \,,
  &\qquad&
           x\in\RN
           \,.
\end{alignat}
Here $x=(x_{1},\dots,x_{N})$, $\grad u$ [respectively, $\Div$] is the gradient [respectively, the divergence] with respect to $x$, and we denote
\begin{equation}
  \label{eq:expweight}
  \ew(x)
  =
  e^{\exw(\abs{x})}
  \,,
  \qquad
  x\in\RN
  \,.
\end{equation}
We assume throughout, often without further reference, that $u\ge 0$, $u_{0}\ge 0$, $u_{0}\ew\in L^{1}(\RN)$, that
\begin{equation}
  \label{eq:degen}
  1<p<N
  \,,
  \qquad
  p+m-3>0
  \,,
\end{equation}
and that $\exw\in C([0,+\infty])\cap C^{1}((0,+\infty))$ is such that $\exw(0)=0$, $\exw(s)>0$ for $s>0$, and for given $0<\alpha_{1}\le \alpha_{2}<\min(N,p/(p-1))$
\begin{equation}
  \label{eq:powerlike}
  \alpha_{1}
  \frac{\exw(s)}{s}
  \le
  \exw'(s)
  \le
  \alpha_{2}
  \frac{\exw(s)}{s}
  \,,
  \qquad
  s>0
  \,.
\end{equation}
For example the power $\exw(s)=s^{\alpha}$ for $\alpha$ as above satisfies our assumptions, as well as the Zygmund function
\begin{equation}
  \label{eq:zygmund}
  \exw(s)
  =
  s^{\alpha}
  [\log(c+s)]^{\beta}
  \,,
  \quad
  s\ge 0
  \,,
  c>1
  \,,
  \alpha>0
  \,,
  \beta>0
  \,;
\end{equation}
here $\alpha_{1}=\alpha$, $\alpha_{2}=\alpha+\beta$.

Next we give a formal definition of the notion of weak solution to
\eqref{eq:pde}--\eqref{eq:initd}. For an open domain $\Omega\subset\RN$, and $q\in[1,+\infty)$, we
denote by $W_{\ew}^{1,q}(\Omega)$ the weighted Sobolev space with norm
\begin{equation}
  \label{eq:Wf}
  \norma{w}{W_{\ew}^{1,q}(\Omega)}
  :=
  \Big(
  \int_{\Omega}
  (
  \abs{\grad w}^{q}
  +
  \abs{w}^{q}
  )
  \di\ew
  \Big)^{\frac{1}{q}}
  \,,
  \quad
  \di\ew
  =
  \ew(x)
  \di x
  \,.
\end{equation}
We often denote norms in $L^{q}(\RN)$ without reference to
the domain, as in $\norma{F}{q}=\norma{F}{L^{q}(\RN)}$.
The weighted spaces $L^{q}_{\ew}(\RN)$ are defined in the obvious way (similarly to \eqref{eq:Wf}).
We denote by $B_{R}$ the ball of center $0$ and radius $R$.

By means of the change of variable $v=u^{1/\beta}$, $\beta=(p-1)/(p+m-2)$ we transform \eqref{eq:pde} into
\begin{equation}
  \label{eq:pde_v}
  \ew(x)
  \pder{v^{\beta}}{t}
  -
  \beta^{p-1}
  \Div\big(
  \ew(x)
  \abs{\grad v}^{p-2}
  \grad v
  \big)
  =
  0
  \,,
  \qquad
  \text{in $S_{T}$.}
\end{equation}
\begin{definition}
  \label{d:weakon}
  We say that $u\ge 0$ is a weak solution to \eqref{eq:pde}--\eqref{eq:initd} if
  \begin{equation*}
    v=u^{1/\beta}
    \in
    C((0,T);L_{\ew}^{1+\beta}(\RN))
    \cap
    L_{\textup{loc}}^{p}
    ((0,T);W_{\ew}^{1,p}(\RN))
    \cap
    L_{\textup{loc}}^{\infty}(S_{T})
    \,,
  \end{equation*}
  and for any $0<t_{1}<t_{2}<T$ and any test function
  \begin{equation*}
    \eta
    \in
    W_{\textup{loc}}^{1,2}((0,T);L_{\ew}^{1+\beta}(\RN))
    \cap
    L_{\textup{loc}}^{p}
    ((0,T);W_{\ew}^{1,p}(\RN))
    \,,
  \end{equation*}
  we have
  \begin{equation}
    \label{eq:weakon_a}
    \begin{split}
      &\int_{\RN}
        [
        v^{\beta}(t_{2})
        \eta(t_{2})
        -
        v^{\beta}(t_{1})
        \eta(t_{1})
        ]
        \di\ew
      \\
      &\quad
        =
        \int_{t_{1}}^{t_{2}}
        \int_{\RN}
        \Big[
        -v^{\beta}
        \pder{\eta}{t}
        +
        \beta^{p-1}
        \abs{\grad v}^{p}
        \grad v
        \grad\eta
        \Big]
        \di\ew
        \di t
        =
        0
        \,.
    \end{split}
  \end{equation}
  Moreover $u$ takes the initial data in the sense that for any $\zeta\in C_{0}(\RN)$
  \begin{equation}
    \label{eq:weakon_b}
    \lim_{t\to 0}
        \int_{\RN}
        u(x,t)
        \zeta(x)
        \di\ew
        =
        \int_{\RN}
        u_{0}(x)
        \zeta(x)
        \di\ew
        \,.
  \end{equation}
\end{definition}

The existence of weak solutions to our Cauchy problem can be established by well-known techniques, see e.g., \cite{Andreucci:Tedeev:2022b}.

In this paper we address the following issues:
\\
In the case $\alpha_{2}<1$, we obtain a precise $\sup$ bound for solutions to \eqref{eq:pde} which are either radial or compactly supported (in space).
\\
For any $0<\alpha_{2}<\min(N,p/(p-1))$ we prove the property of finite speed of propagation, that is that solutions with a compactly supported initial data maintain their support bounded for all $t>0$, and we give an estimate of its size.

More exactly our results are contained in the following theorems.

\begin{theorem}
  \label{t:sur}
  Let $u$ be a non-negative radial weak solution of \eqref{eq:pde}--\eqref{eq:initd} in $S_{\infty}$. Assume also
  \begin{equation}
    \label{eq:sur_n}
    1>\alpha_{2}\ge \alpha_{1}\ge \frac{\alpha_{2}}{\alpha_{2}+1}
    \,.
  \end{equation}
  Then for large $t$
  \begin{equation}
    \label{eq:sur_nn}
    \norma{u(t)}{\infty}
    \le
    c
    \Big[
    \frac{
      \exw^{(-1)}(
      \log (t\norma{u_{0}\ew}{1}^{p+m-3})
      )^{p}
    }{
      \log (t\norma{u_{0}\ew}{1}^{p+m-3})
    }
    \Big]^{\frac{1}{p+m-3}}
    t^{-\frac{1}{p+m-3}}
    \norma{u_{0}\ew}{1}^{-1}
    \,.
  \end{equation}
Here $c=c(N,p,m,\alpha_{1},\alpha_{2},\exw(1))$.
\end{theorem}

\begin{theorem}
  \label{t:bds}
  Let $u$ be a non-negative weak solution of \eqref{eq:pde}--\eqref{eq:initd} in $S_{\infty}$. Assume also \eqref{eq:sur_n}, and that the initial data is compactly supported.
  \\
  Then for large $t$, the same bound as in \eqref{eq:sur_nn} is in force.
\end{theorem}

\begin{theorem}
  \label{t:fsp}
  Let $u$ be a non-negative weak solution of \eqref{eq:pde}--\eqref{eq:initd} in $S_{\infty}$. Assume also that the initial data is compactly supported.
  \\
  Then for all $t>0$ we have $\supp u(t)\subset B_{R(t)}$, where
  \begin{equation}
    \label{eq:fsp_n}
    R(t)
    =
    c
    \exw^{(-1)}
    \big(
    \log(e+t\norma{u_{0}\ew}{1}^{p+m-3})
    \big)
    \,.
  \end{equation}
  Here $c=c(N,p,m,\alpha_{1},\alpha_{2},\exw(1),R_{0})$.
\end{theorem}
The assumptions of Theorems \ref{t:sur} and \ref{t:bds} may be weakened, see Remark~\ref{r:aux_alpha}.

\subsection{Example: power function}
\label{s:example1}

Assume here $\exw(s)=s^{\alpha}$, $\alpha>0$.
\\
If $\alpha\in (0,1)$, then the $\sup$ bound \eqref{eq:sur_nn} amounts to
\begin{equation}
  \label{eq:example1_a}
  \norma{u(t)}{\infty}
  \le
  c
  [\log (t\norma{u_{0}\ew}{1}^{p+m-3})
  )]^{\frac{p-\alpha}{\alpha(p+m-3)}}
  t^{-\frac{1}{p+m-3}}
  \norma{u_{0}\ew}{1}^{-1}
  \,.
\end{equation}
For any $\alpha>0$, the definition \eqref{eq:fsp_n} becomes
\begin{equation}
  \label{eq:example1_b}
  R(t)
  =
  c
  \big[
  \log(e+t\norma{u_{0}\ew}{1}^{p+m-3})
  \big]^{\frac{1}{\alpha}}
  \,.
\end{equation}

\subsection{Example: Zygmund function}
\label{s:example2}

Let $\exw$ be as in \eqref{eq:zygmund}. Let $\tau=\exw(s)$. Then
\begin{equation}
  \label{eq:example2_a}
  s
  =
  \exw^{(-1)}(\tau)
  =
  \alpha^{\frac{\beta}{\alpha}}
  \tau^{\frac{1}{\alpha}}
  (\log\tau)^{-\frac{\beta}{\alpha}}
  A(\tau)
  \,,
  \qquad
  \tau\to+\infty
  \,,
\end{equation}
where $A(\tau)=1+\smo(1)$.
Indeed, let us identify $A(\tau)$ by applying $\exw$ to both sides of \eqref{eq:example2_a}; we obtain
\begin{equation}
  \label{eq:example2_b}
  \begin{split}
    \tau
    &=
      \alpha
      \tau
      (\log \tau)^{-\beta}
      A(\tau)^{\alpha}
      \\
    &\quad\times
      \Big[
      \frac{\beta}{\alpha}
      \log\alpha
      +
      \frac{1}{\alpha}
      \log \tau
      -
      \frac{\beta}{\alpha}
      \log\log\tau
      +
      \log A(\tau)
      \Big]^{\beta}
      \\
    &=
      \tau
      A(\tau)^{\alpha}
      \Big[
      1
      +
      \smo(1)
      +
      \alpha
      \frac{\log A(\tau)}{\log\tau}
      \Big]^{\beta}
      \,,
  \end{split}
\end{equation}
whence necessarily $A(\tau)\to 1$ as $\tau\to+\infty$.

Let us assume for the sake of notational simplicity that $\norma{u_{0}\ew}{1}=1$. Then, if $\alpha_{2}=\alpha+\beta<1$ and
\begin{equation*}
  \alpha
  >
  \frac{\alpha+\beta}{\alpha+\beta+1}
  \,,
\end{equation*}
it follows from \eqref{eq:sur_n} that
\begin{equation}
  \label{eq:example2_sup}
  \norma{u(t)}{\infty}
  \le
  c
  \Big[
  \frac{1}{\log t}
  \Big(
  \frac{
    \log t
  }{
    (\log\log t)^{\beta}
  }
  \Big)^{\frac{p}{\alpha}}
  \Big]^{\frac{1}{p+m-3}}
  t^{-\frac{1}{p+m-3}}
  \,.
\end{equation}
In addition, for any $\alpha$, $\beta>0$, \eqref{eq:fsp_n} yields for large $t$
\begin{equation}
  \label{eq:example2_c}
  R(t)
  \le
  c
  \Big(
  \frac{
    \log t
  }{
    (\log\log t)^{\beta}
  }
  \Big)^{\frac{1}{\alpha}}
  \,.
\end{equation}

\begin{remark}
  \label{r:pme}
  If $p=2$, estimate \eqref{eq:example1_a} is in agreement with the
  one obtained in \cite{Muratori:Roncoroni:2022} for the porous media equation in Cartan-Hadamard manifolds. See also the related papers \cite{Grillo:Muratori:2016}, \cite{Grillo:Muratori:Vazquez:2017}, \cite{Muratori:2021} and references therein. However, it seems to us that the estimate \eqref{eq:example2_sup} is new even for the porous media equation.
\end{remark}

\begin{remark}
  \label{r:partic}
  In the case $m=1$, $\exw(s)=s^{\alpha}$, Theorem~\ref{t:fsp} was proved in \cite{Sanikidze:Tedeev:2010}. The case $\alpha_{1}=\alpha_{2}=1$ was considered in \cite{Tedeev:2023}.
\end{remark}

Weighted equations like \eqref{eq:pde} play an important role in Riemannian geometry, see \cite{Bakry:Emery:1985}, \cite{Grigoryan:1999b}, \cite{Grigoryan:2006}. Note that the precise form of exponentially weighted Sobolev inequalities is not yet completely understood, in spite of their relevance to the study of qualitative asymptotic behavior of solutions to diffusion equations. Such inequalities are connected with isoperimetric inequalities (see \cite{Alvino:etal:2017}, \cite{Andreucci:Tedeev:2021c}, \cite{Betta:Brock:Mercaldo:Posteraro:1999}, \cite{Brandolini:Chiacchio:2023}, \cite{Chambers:2019}, \cite{Figalli:Maggi:2013}, \cite{Fusco:LaManna:2023}, \cite{Kolesnikov:Zhdanov:2011} and references therein).
\\
In the study of the porous media equation in Cartan-Hadamard manifolds carried out in \cite{Grillo:Muratori:2016}, \cite{Grillo:Muratori:Vazquez:2017}, \cite{Muratori:Roncoroni:2022},  one of the main ingredients is a precise version of a weighted Sobolev inequality connected with the volume growth rate, depending on the behavior of sectional curvature. As it was shown there the most non-standard cases are connected with manifolds with exponentially growing volume. In that case classical approaches fail to prove suitable weighted Sobolev inequalities. The first result in this direction \cite{Grillo:Muratori:2016} deals with Cartan-Hadamard manifolds admitting exponentially weighted global Poincar\'{e} inequalities. Later in \cite{Muratori:Roncoroni:2022} such results were extended for various classes of Cartan-Hadamard manifolds. We also recall that, still in the setting of the Cauchy problem for the porous media equation in Cartan-Hadamard manifolds with exponential volume growth, in \cite{Grillo:Muratori:Vazquez:2017} the authors obtained the complete classification of time decay rates, by employing self-similar sub- and super-solutions.
\\
The radial Sobolev inequality obtained in \cite{Muratori:Roncoroni:2022} is enough to get precise $\sup$ estimates in Cartan-Hadamard manifolds. The radial inequality we prove here (Lemma~\ref{l:aux_hardy}) is a natural generalization of the one in \cite{Muratori:Roncoroni:2022}; also our proof follows their approach. The extension of Theorem~\ref{t:sur} to the non-radial case is an open problem. In this connection, here we consider the case of compactly supported solutions, proving the new Sobolev inequality in Lemma~\ref{l:aux_bdd} which enables us to prove Theorem~\ref{t:bds}. 
\\
Weighted parabolic equations were studied also in \cite{Andreucci:Cirmi:Leonardi:Tedeev:2001}, \cite{Andreucci:Tedeev:2021b}, \cite{Andreucci:Tedeev:2022b}, \cite{Grigoryan:2006}, \cite{Grigoryan:Surig:2024}, \cite{Ohya:2004}, \cite{Tedeev:2007}, \cite{Tedeev:2023}.

In the proofs of Theorems \ref{t:sur} and \ref{t:bds} we use the same DeGiorgi approach as in \cite{Andreucci:Tedeev:2005}, \cite{Andreucci:Tedeev:2015}, while in the proof of the finite speed of propagation, i.e., of Theorem~\ref{t:fsp}, we follow the approach in \cite{Andreucci:Tedeev:1999}, \cite{Andreucci:Tedeev:2021b}. The latter approach needs only the Poincar\'{e} type estimate in Lemma~\ref{l:aux_poinc}, which allows us to skip the assumption $\alpha_{2}<1$ in this case.
\\
Hopefully the results of this paper can be applied to the qualitative investigation of doubly non-linear degenerate parabolic problems in Cartan-Hadamard manifolds.

The paper is organized as follows: Section~\ref{s:aux} is devoted to auxiliary results and functional inequalities; Section~\ref{s:sur} [respectively \ref{s:bds}, \ref{s:fsp}] contains the proof of Theorem~\ref{t:sur} [respectively \ref{t:bds}, \ref{t:fsp}]; note that Theorem~\ref{t:fsp} is used in the proof of Theorem~\ref{t:bds}, but its proof is of course completely independent of the latter Theorem.

\section{Auxiliary results}
\label{s:aux}

We remark first the following elementary consequences of \eqref{eq:powerlike}:
\begin{gather}
  \label{eq:aux_powerlike_inv}
  \exw^{(-1)}(z)
  \lambda^{\frac{1}{\alpha_{2}}}
  \le
  \exw^{(-1)}(z\lambda)
  \le
  \exw^{(-1)}(z)
  \lambda^{\frac{1}{\alpha_{1}}}
  \,,
  \quad
  z>0
  \,,
  \lambda>1
  \,,
  \\
  \label{eq:aux_powerlike_inv2}
  \exw^{(-1)}(z)
  \lambda^{\frac{1}{\alpha_{1}}}
  \le
  \exw^{(-1)}(z\lambda)
  \le
  \exw^{(-1)}(z)
  \lambda^{\frac{1}{\alpha_{2}}}
  \,,
  \quad
  z>0
  \,,
  \lambda<1
  \,.
\end{gather}
With the purpose of introducing a more regular version of $\exw$, we define the function
\begin{equation*}
  \emf(s)
  =
  \frac{1}{s}
  \int_{0}^{s}
  \exw(z)
  \di z
  \,,
  \qquad
  s>0
  \,.
\end{equation*}
We have
\begin{align}
  \label{eq:aux_emf_prime}
  \emf'(s)
  &=
    -
    \frac{1}{s^{2}}
    \int_{0}^{s}
    \exw(z)
    \di z
    +
    \frac{\exw(s)}{s}
    \,,
  \\
  \label{eq:aux_emf_second}
  \emf''(s)
  &=
    \frac{1}{s^{2}}
    \Big(
    \frac{2}{s}
    \int_{0}^{s}
    \exw(z)
    \di z
    -
    2\exw(s)
    +
    s\exw'(s)
    \Big)
    \,.
\end{align}
According to our assumption \eqref{eq:powerlike}, we see by direct differentiation that
\begin{equation}
  \label{eq:aux_non_non}
  \begin{aligned}
    &\text{$s\mapsto \exw(s)s^{-\alpha_{1}}$ is non-decreasing and}
      \\
    &\text{$s\mapsto \exw(s)s^{-\alpha_{2}}$ is non-increasing in $(0,+\infty)$.}
  \end{aligned}
\end{equation}
 Thus we have
\begin{equation}
  \label{eq:aux_emf_int}
  \begin{split}
    &\frac{\exw(s)s}{\alpha_{2}+1}
    =
    \exw(s)s^{-\alpha_{2}}
    \frac{s^{\alpha_{2}+1}}{\alpha_{2}+1}
    \le
    \int_{0}^{s}
    \exw(z)z^{-\alpha_{2}}
    z^{\alpha_{2}}
    \di z
    =
    \int_{0}^{s}
    \exw(z)
    \di z
    \\
    &\quad=
    \int_{0}^{s}
    \exw(z)z^{-\alpha_{1}}
    z^{\alpha_{1}}
    \di z
    \le
    \exw(s)s^{-\alpha_{1}}
    \frac{s^{\alpha_{1}+1}}{\alpha_{1}+1}
    =
    \frac{\exw(s)s}{\alpha_{1}+1}
    \,.
  \end{split}
\end{equation}
As a first consequence of \eqref{eq:aux_emf_int} and of \eqref{eq:aux_emf_prime} we infer for all $s>0$
\begin{equation}
  \label{eq:aux_emf_prpos}
  0
  <
  \frac{\alpha_{1}}{\alpha_{1}+1}
  \frac{\exw(s)}{s}
  \le
  \emf'(s)
  \le
  \frac{\alpha_{2}}{\alpha_{2}+1}
  \frac{\exw(s)}{s}
  \,.
\end{equation}

\begin{lemma}
  \label{l:aux_emf}
  1) Assume
    \begin{equation}
    \label{eq:aux_emf_nn}
    \frac{(N-p)\alpha_{1}}{\alpha_{1}+1}
    +
    (p-1)
    \Big(
    \alpha_{1}
    -
    \frac{\alpha_{2}}{\alpha_{2}+1}
    \Big)
    \ge
    0
    \,.
  \end{equation}
  Then we have
  \begin{equation}
    \label{eq:aux_emf_n}
    \der{}{s}
    \big(
    \emf'(s)^{p-1}
    s^{N-1}
    \big)
    \ge 0
    \,,
  \end{equation}
  as well as
  \begin{equation}
    \label{eq:aux_emf_m}
    \der{}{s}
    \big(
    \emf'(s)^{-1}
    s^{-\frac{N-1}{p-1}}
    \big)
    \le 0
    \,.
  \end{equation}
  2) If instead we assume
  \begin{equation}
    \label{eq:aux_emf_p}
    \frac{(N+1)\alpha_{1}}{\alpha_{1}+1}
    \ge
    \alpha_{2}
    \,,
  \end{equation}
  we have
  \begin{equation}
    \label{eq:aux_emf_pp}
    \der{}{s}
    \big(
    \emf'(s)^{-1}
    s^{N-1}
    \big)
    \ge
    0
    \,.
  \end{equation}
\end{lemma}

It is easily seen that \eqref{eq:aux_emf_nn} and \eqref{eq:aux_emf_p} are satisfied if $\alpha_{1}=\alpha_{2}$, and also in a suitable range of $\alpha_{1}<\alpha_{2}$. For example this is the case if
\begin{equation}
  \label{eq:aux_alpha_assump}
  1\ge
  \alpha_{2}
  >
  \alpha_{1}
  \ge
  \frac{\alpha_{2}}{\alpha_{2}+1}
  \,,
  \qquad
  N\ge 2
  \,.
\end{equation}

\begin{proof}
  1) We calculate
  \begin{equation}
    \label{eq:aux_emf_i}
    \der{}{s}
    \big(
    \emf'(s)^{p-1}
    s^{N-1}
    \big)
    =
    s^{N-2}
    \emf'(s)^{p-2}
    A_{1}(s)
    \,,
  \end{equation}
  where
  \begin{equation*}
    A_{1}(s)
    :=
    (p-1)
    s
    \emf''(s)
    +
    (N-1)
    \emf'(s)
    \,.
  \end{equation*}
  According to \eqref{eq:aux_emf_prime}, \eqref{eq:aux_emf_second} and \eqref{eq:aux_emf_int} we estimate
  \begin{equation*}
    \begin{split}
      A_{1}(s)s^{2}
      &=
        (N-p)
        \Big[
        \exw(s)
        s
        -
        \int_{0}^{s}
        \exw(z)
        \di z
        \Big]
        \\
      &\quad
        +
        (p-1)
        \Big[
        \int_{0}^{s}
        \exw(z)
        \di z
        +
        s^{2}
        \exw'(s)
        -
        \exw(s)
        s
        \Big]
        \\
      &\ge
        \exw(s)s
        \Big[
        \frac{(N-p)\alpha_{1}}{\alpha_{1}+1}
        +
        (p-1)
        \Big(
        \alpha_{1}
        -
        \frac{\alpha_{2}}{\alpha_{2}+1}
        \Big)
        \Big]
        \ge
        0
        \,,
    \end{split}
  \end{equation*}
  on invoking our assumption \eqref{eq:aux_emf_nn}. The statement in \eqref{eq:aux_emf_n} is proved.
  \\
  Then in order to prove \eqref{eq:aux_emf_m} we only need observe that
  \begin{equation*}
    \emf'(s)^{-1}
    s^{-\frac{N-1}{p-1}}
    =
    \Big(
    \emf'(s)^{p-1}
    s^{N-1}
    \Big)^{-\frac{1}{p-1}}
    \,.
  \end{equation*}

2) We calculate
  \begin{equation}
    \label{eq:aux_emf_j}
    \der{}{s}
    \big(
    \emf'(s)^{-1}
    s^{N-1}
    \big)
    =
    s^{N-2}
    \emf'(s)^{-2}
    A_{2}(s)
    \,,
  \end{equation}
  where
  \begin{equation*}
    A_{2}(s)
    :=
    (N-1)
    \emf'(s)
    -
    s
    \emf''(s)
    \,.
  \end{equation*}
  According to \eqref{eq:aux_emf_prime}, \eqref{eq:aux_emf_second} and \eqref{eq:aux_emf_int} we estimate
  \begin{equation*}
    \begin{split}
      A_{2}(s)s^{2}
      &=
        (N+1)
        \Big[
        \exw(s)
        s
        -
        \int_{0}^{s}
        \exw(z)
        \di z
        \Big]
        -
        s^{2}\exw'(s)
        \\
      &\ge
        s\exw(s)
        \Big[
        \frac{(N+1)\alpha_{1}}{\alpha_{1}+1}
        -
        \alpha_{2}
        \Big]
        \ge
        0
        \,,
    \end{split}
  \end{equation*}
  on invoking our assumption \eqref{eq:aux_emf_p}. The statement in \eqref{eq:aux_emf_pp} is proved.
\end{proof}

\begin{remark}
  \label{r:aux_alpha}
  In the rest of the paper we are going to apply \eqref{eq:aux_emf_n}, \eqref{eq:aux_emf_m}, respectively \eqref{eq:aux_emf_pp}, so assumptions \eqref{eq:aux_emf_nn}, respectively \eqref{eq:aux_emf_p}, may be replaced directly by these estimates.
  \\
  Concerning the assumptions of Theorem~\ref{t:sur}, we certainly need $\alpha_{2}<1$, but the inequality $\alpha_{1}\ge \alpha_{2}/(\alpha_{2}+1)$ is only used to guarantee \eqref{eq:aux_emf_nn} as well as \eqref{eq:aux_emf_p}.
\end{remark}

In the following we use the notation $\fkf(s)=\emf'(s)$, $s>0$, $\fkf(0)=0$.

\begin{lemma}
  \label{l:aux_poinc}
  Assume $v\in W^{1,p}_{\ew}(\RN)$, $N>p>1$, and \eqref{eq:aux_emf_nn}.
  Then
  \begin{equation}
    \label{eq:aux_poinc_n}
    \int_{\RN}
    \fkf(\abs{x})^{p}
    \abs{v(x)}^{p}
    \di\ew
    \le
    C
    \int_{\RN}
    \abs{\grad v(x)}^{p}
    \di\ew
    \,.
  \end{equation}
  Here $C=C(N,p,\alpha_{1},\alpha_{2})$.
\end{lemma}

\begin{proof}
  We start by proving the following one dimensional version of our claim
  \begin{equation}
    \label{eq:aux_poinc_i}
    \int_{0}^{+\infty}
    \fkf(t)^{p}
    \abs{v(t)}^{p}
    t^{N-1}
    e^{\exw(t)}
    \di t
    \le
    C
    \int_{0}^{+\infty}
    \abs{v'(t)}^{p}
    t^{N-1}
    e^{\exw(t)}
    \di t
    \,,
  \end{equation}
  for all compactly supported $v\in C^{1}([0,+\infty))$. We borrow from \cite[Theorem~6.2]{Opic:Kufner:hardy} the following Hardy-type inequality: let  $1\le p\le q<\infty$, let $\varphi$, $w$ be positive and locally integrable in $(0,+\infty)$. Let also $v$ belong to the set of absolutely continuous functions over $[0,+\infty)$ vanishing at $\infty$, and assume that
  \begin{equation*}
    \beta_{R}
    :=
    \sup_{0<r<+\infty}
    \norma{w^{\frac{1}{q}}}{L^{q}((0,r))}
    \norma{\varphi^{-\frac{1}{p}}}{L^{p'}((r,+\infty))}
    <
    +
    \infty
    \,.
  \end{equation*}
  Then 
  \begin{equation}
    \label{eq:aux_poinc_hardy}
    \Big(
    \int_{0}^{+\infty}
    \abs{v(r)}^{q}
    w(r)
    \di r
    \Big)^{\frac{1}{q}}
    \le
    C_{R}
    \Big(
    \int_{0}^{+\infty}
    \abs{v'(r)}^{p}
    \varphi(r)
    \di r
    \Big)^{\frac{1}{p}}
    \,,
  \end{equation}
  where the optimal constant $C_{R}$ satisfies
  \begin{equation}
    \label{eq:aux_poinc_const}
    \beta_{R}
    \le
    C_{R}
    \le
    K(q,p)
    \beta_{R}
    \,,
    \quad
    K(q,p)
    :=
    \Big(
    1
    +
    \frac{q}{p'}
    \Big)^{\frac{1}{q}}
    \Big(
    1
    +
    \frac{p'}{q}
    \Big)^{\frac{1}{p'}}
    \,.
  \end{equation}
  We choose here $q=p$,
  \begin{equation*}
    w(r)
    =
    \fkf(r)^{p}
    r^{N-1}
    e^{\exw(r)}
    \,,
    \qquad
    \varphi(r)
    =
    r^{N-1}
    e^{\exw(r)}
    \,.
  \end{equation*}
  With this choice, \eqref{eq:aux_poinc_i} coincides with \eqref{eq:aux_poinc_hardy}.
  Thus we only need to check that
  \begin{equation}
    \label{eq:aux_poinc_ii}
    \begin{split}
      \beta_{R}
      &=
        \sup_{0<r<+\infty}
        \Big(
        \int_{0}^{r}
        \fkf(z)^{p}
        z^{N-1}
        e^{\exw(z)}
        \di z
        \Big)^{\frac{1}{p}}
      \\
      &\quad\times
        \Big(
        \int_{r}^{+\infty}
        z^{-\frac{N-1}{p-1}}
        e^{-\frac{\exw(z)}{p-1}}
        \di z
        \Big)^{\frac{p-1}{p}}
        =:
        I_{1}^{\frac{1}{p}}
        I_{2}^{\frac{p-1}{p}}
        <
        +\infty
        \,.
    \end{split}
  \end{equation}
  On applying in sequence \eqref{eq:aux_emf_prime}, \eqref{eq:powerlike}, integration by parts and then \eqref{eq:aux_emf_n}, we get
  \begin{equation*}
    \begin{split}
      I_{1}
      &=
        \int_{0}^{r}
        \fkf(z)^{p-1}
        z^{N-1}
        \frac{\emf'(z)}{\exw'(z)}
        \exw'(z)
        e^{\exw(z)}
        \di z
        \le
        \frac{1}{\alpha_{1}}
        \int_{0}^{r}
        \fkf(z)^{p-1}
        z^{N-1}
        \exw'(z)
        e^{\exw(z)}
        \di z
        \\
      &=
        \frac{1}{\alpha_{1}}
        \fkf(r)^{p-1}
        r^{N-1}
        e^{\exw(r)}
        -
        \int_{0}^{r}
        \der{}{z}
        \big(
        \fkf(z)^{p-1}
        z^{N-1}
        \big)
        e^{\exw(z)}
        \di z
        \\
      &\le
        \frac{1}{\alpha_{1}}
        \fkf(r)^{p-1}
        r^{N-1}
        e^{\exw(r)}
        \,.
    \end{split}
  \end{equation*}
  Next, by means of a similar argument, we obtain by invoking in sequence \eqref{eq:powerlike}, \eqref{eq:aux_emf_prpos}, integration by parts and then \eqref{eq:aux_emf_m}, 
  \begin{equation*}
    \begin{split}
      I_{2}
      &=
        -
        (p-1)
        \int_{r}^{+\infty}
        z^{-\frac{N-1}{p-1}}
        \frac{1}{\exw'(z)}
        \der{e^{-\frac{\exw(z)}{p-1}}}{z}
        \di z
        \le
        -
        c_{1}
        \int_{r}^{+\infty}
        z^{-\frac{N-1}{p-1}}
        \frac{1}{\fkf(z)}
        \der{e^{-\frac{\exw(z)}{p-1}}}{z}
        \di z
        \\
      &=
        c_{1}
        r^{-\frac{N-1}{p-1}}
        \frac{1}{\fkf(r)}
        e^{-\frac{\exw(r)}{p-1}}
        +
        c_{1}
        \int_{r}^{+\infty}
        \der{}{z}
        \Big(
        z^{-\frac{N-1}{p-1}}
        \frac{1}{\fkf(z)}
        \Big)
        e^{-\frac{\exw(z)}{p-1}}
        \di z
        \\
      &\le
        c_{1}
        r^{-\frac{N-1}{p-1}}
        \frac{1}{\fkf(r)}
        e^{-\frac{\exw(r)}{p-1}}
        \,,
        \qquad
        c_{1}
        :=
        \frac{(p-1)\alpha_{1}}{\alpha_{2}(\alpha_{1}+1)}
        \,.
    \end{split}
  \end{equation*}
  Thus
  \begin{equation*}
    \begin{split}
      \beta_{R}
      &\le
        \Big(
        \frac{1}{\alpha_{1}}
        \fkf(r)^{p-1}
        r^{N-1}
        e^{\exw(r)}
        \Big)^{\frac{1}{p}}
        \Big(
        c_{1}
        r^{-\frac{N-1}{p-1}}
        \frac{1}{\fkf(r)}
        e^{-\frac{\exw(r)}{p-1}}
        \Big)^{\frac{p-1}{p}}
      \\
      &=
        \Big(
        \frac{c_{1}^{p-1}}{\alpha_{1}}
        \Big)^{\frac{1}{p}}
        \,.
    \end{split}
  \end{equation*}

  Finally we can reduce the general case to the one dimensional inequality
  \eqref{eq:aux_poinc_i}.
  Indeed, the inequality \eqref{eq:aux_poinc_n} is easily obtained by writing any $v\in C^{1}_{0}(\RN)$ in polar coordinates as $v(x)=\widetilde v(r,\theta)$, $r>0$, $\theta:=(\theta_{1},\dots,\theta_{N-1})\in \varOmega$ so that $\di x=r^{N-1}\varTheta(\theta)\di r\di \theta$, for a suitable $\varTheta(\theta)>0$, and on invoking \eqref{eq:aux_poinc_i} we find
  \begin{multline}
    \label{eq:aux_poinc_j}
    \int_{\RN}
    \fkf(\abs{x})^{p}
    e^{\exw(\abs{x})}
    \abs{v(x)}^{p}
    \di x
    =
    \int_{\varOmega}
    \Big(
    \int_{0}^{+\infty}
    \fkf(r)^{p}
    e^{\exw(r)}
    r^{N-1}
    \abs{\widetilde v(r,\theta)}^{p}
    \di r
    \Big)
    \varTheta(\theta)
    \di \theta
    \\
    \le
    C
    \int_{\varOmega}
    \Big(
    \int_{0}^{+\infty}
    e^{\exw(r)}
    r^{N-1}
    \abs{\widetilde v_{r}(r,\theta)}^{p}
    \di r
    \Big)
    \varTheta(\theta)
    \di \theta
    \le
    C
    \int_{\RN}
    \abs{\grad v(x)}^{p}
    e^{\exw(x)}
    \di x
    \,.
  \end{multline}
\end{proof}

In the following we denote $a=pq/(q-p)$; note that assumption \eqref{eq:aux_hardy_nn} is equivalent to $a>N$.

\begin{lemma}
  \label{l:aux_hardy}
  Let $v:(0,+\infty)\to\R$ be absolutely continuous on each compact interval contained in $(0,+\infty)$, and such that $v(r)\to 0$ as $r\to+\infty$ as well as
  \begin{equation}
    \label{eq:aux_hardy_n}
    \int_{0}^{+\infty}
    \abs{v'(r)}^{p}
    r^{N-1}
    e^{\exw(r)}
    \di r
    <
    +\infty
    \,.
  \end{equation}
  Assume also \eqref{eq:aux_emf_nn},  \eqref{eq:aux_emf_p}, $\alpha_{1}\le\alpha_{2}<1$ and
  \begin{equation}
    \label{eq:aux_hardy_nn}
    \frac{Np}{N-p}
    >
    q
    >
    p
    \,.
  \end{equation}
  Then
  \begin{equation}
    \label{eq:aux_hardy_m}
    \Big(
    \int_{0}^{+\infty}
    \abs{v(r)}^{q}
    r^{N-1}
    e^{\exw(r)}
    \di r
    \Big)^{\frac{1}{q}}
    \le
    \CH
    \Big(
    \int_{0}^{+\infty}
    \abs{v'(r)}^{p}
    r^{N-1}
    e^{\exw(r)}
    \di r
    \Big)^{\frac{1}{p}}
    \,.
  \end{equation}
  Here for $K=K(q,p)$ as in \eqref{eq:aux_poinc_const} we have for any given $r_{0}>0$, using the notation $a=pq/(q-p)$, 
  \begin{equation}
    \label{eq:aux_hardy_mm}
    \begin{split}
      \CH
      &=
        K(q,p)
        \Big[
        \frac{1}{N^{\frac{1}{q}}}
        \Big(
        \frac{p-1}{N-p}
        \Big)^{\frac{p-1}{p}}
        (1+r_{0})
      \\
      &+
        (p-1)^{\frac{p-1}{p}}
        \Big(
        \frac{\alpha_{2}(\alpha_{1}+1)}{\alpha_{1}(\alpha_{2}+1)}
        \Big)^{\frac{1}{q}}
        \alpha_{1}^{-\frac{1}{q}}
        \alpha_{2}^{-\frac{p-1}{p}}
      \\
      &\quad\times
        \Big(
        \sum_{i=1}^{2}
        \Big(
        \frac{1}{\alpha_{i}}
        -
        1
        \Big)^{\frac{1}{\alpha_{i}}-1}
        e^{-(\frac{1}{\alpha_{i}}-1)}
        \Big)
        \Big(
        1
        +
        \frac{\exw(r_{0})^{\frac{1}{N}}}{r_{0}}
        \Big)
        \frac{\exw^{(-1)}(a)}{a}
        \Big]
        \,.
    \end{split}
  \end{equation}
\end{lemma}

\begin{proof}
  We invoke again, from \cite{Opic:Kufner:hardy}, the inequality \eqref{eq:aux_poinc_hardy}, where we take $\varphi(r)=w(r)=r^{N-1}e^{\exw(r)}$, immediately getting \eqref{eq:aux_hardy_m} with $\CH$ formally replaced by the constant
  \begin{equation*}
    C
    =
    \beta_{0}
    K(q,p)
    \,,
  \end{equation*}
  provided $\beta_{0}:=\sup\{A(r):r>0\}<+\infty$, where
  \begin{equation*}
    A(r)
    =
    \Big(
    \int_{0}^{r}
    z^{N-1}
    e^{\exw(z)}
    \di z
    \Big)^{\frac{1}{q}}
    \Big(
    \int_{r}^{+\infty}
    z^{-\frac{N-1}{p-1}}
    e^{-\frac{\exw(z)}{p-1}}
    \di z
    \Big)^{\frac{p-1}{p}}
    =:
    J_{1}^{\frac{1}{q}}
    J_{2}^{\frac{p-1}{p}}
    \,.
  \end{equation*}
  We then only have to find a bound for $\beta_{0}$.
  \\
  On using the monotonicity of $\exw$, and of course of the exponential, and the fact that $(N-1)/(p-1)>1$ we find by means of routine majorizations for all $r>0$
  \begin{equation}
    \label{eq:aux_hardy_i}
    \begin{split}
    A(r)
    &\le
    \Big(
    \frac{r^{N}}{N}
    e^{\exw(r)}
    \Big)^{\frac{1}{q}}
    \Big(
    \frac{p-1}{N-p}
    r^{-\frac{N-p}{p-1}}
    e^{-\frac{\exw(r)}{p-1}}
    \Big)^{\frac{p-1}{p}}
    \\
      &=
        \frac{1}{N^{\frac{1}{q}}}
        \Big(
        \frac{p-1}{N-p}
        \Big)^{\frac{p-1}{p}}
        r^{\frac{Np-q(N-p)}{qp}}
        e^{-\exw(r)\frac{q-p}{pq}}
        \,.
    \end{split}
  \end{equation}
Note that \eqref{eq:aux_hardy_i} yields an uniform bound for $A$ in all $(0,+\infty)$, since $Np-q(N-p)>0$ according to our assumptions. However, this bound is not optimal for large $r$ and therefore we are going to use it only for $0<r\le r_{0}$, for any fixed $r_{0}>0$, obtaining for $r\le r_{0}$
\begin{equation}
  \label{eq:aux_hardy_ik}
  A(r)
  \le
  \frac{1}{N^{\frac{1}{q}}}
  \Big(
  \frac{p-1}{N-p}
  \Big)^{\frac{p-1}{p}}
  r_{0}^{1-\frac{N}{a}}
  \le
  \frac{1}{N^{\frac{1}{q}}}
  \Big(
  \frac{p-1}{N-p}
  \Big)^{\frac{p-1}{p}}
  (1+r_{0})
  \,,
\end{equation}
for all values of $r_{0}>0$, since $a>N$.

Instead for large $r$ we reason as follows.   On applying in sequence \eqref{eq:powerlike}, \eqref{eq:aux_emf_prpos}, integration by parts and then \eqref{eq:aux_emf_pp}, we get
\begin{equation}
  \label{eq:aux_hardy_j}
  \begin{split}
    J_{1}
    &=
      \int_{0}^{r}
      \frac{z^{N-1}}{\exw'(z)}
      \der{e^{\exw(z)}}{z}
      \di z
      \le
      c_{2}
      \int_{0}^{r}
      \frac{z^{N-1}}{\fkf(z)}
      \der{e^{\exw(z)}}{z}
      \di z
    \\
    &=
      c_{2}
      \frac{r^{N-1}}{\fkf(r)}
      e^{\exw(r)}
      -
      c_{2}
      \int_{0}^{r}
      \der{}{z}
      \Big(
      \frac{z^{N-1}}{\fkf(z)}
      \Big)
      e^{\exw(z)}
      \di z
    \\
    &\le
      c_{2}
      \frac{r^{N-1}}{\fkf(r)}
      e^{\exw(r)}
      \,,
      \qquad
      c_{2}
      :=
      \frac{\alpha_{2}}{\alpha_{1}(\alpha_{2}+1)}
      \,.
  \end{split}
\end{equation}
Next we remark that, according to the notation in the proof of Lemma~\ref{l:aux_poinc}, we have $J_{2}=I_{2}$ and that we already obtained there the estimate (under assumption \eqref{eq:aux_emf_nn})
\begin{equation}
  \label{eq:aux_hardy_jj}
  J_{2}
  =
  I_{2}
  \le
  c_{1}
  r^{-\frac{N-1}{p-1}}
  \frac{1}{\fkf(r)}
  e^{-\frac{\exw(r)}{p-1}}
  \,,
  \qquad
  c_{1}
  :=
  \frac{(p-1)\alpha_{1}}{\alpha_{2}(\alpha_{1}+1)}
  \,.
\end{equation}
Then we collect \eqref{eq:aux_hardy_j}, \eqref{eq:aux_hardy_jj} and obtain, on invoking \eqref{eq:aux_emf_prpos} again,
\begin{equation}
  \label{eq:aux_hardy_jjj}
  \begin{split}
    A(r)
    &=
      J_{1}^{\frac{1}{q}}
      J_{2}^{\frac{p-1}{p}}
      \le
      c_{1}^{\frac{p-1}{p}}
      c_{2}^{\frac{1}{q}}
      \Big(
      \frac{r^{N-1}}{\fkf(r)}
      e^{\exw(r)}
      \Big)^{\frac{1}{q}}
      \Big(
      r^{-\frac{N-1}{p-1}}
      \frac{1}{\fkf(r)}
      e^{-\frac{\exw(r)}{p-1}}
      \Big)^{\frac{p-1}{p}}
      \\
    &=
      c_{1}^{\frac{p-1}{p}}
      c_{2}^{\frac{1}{q}}
      \Big(
      r^{(N-1)(p-q)}
      \fkf(r)^{-p-q(p-1)}
      e^{(p-q)\exw(r)}
      \Big)^{\frac{1}{pq}}
      \\
    &\le
      c_{3}
      \Big(
      r^{N(p-q)+pq}
      \exw(r)^{-p-q(p-1)}
      e^{(p-q)\exw(r)}
      \Big)^{\frac{1}{pq}}
      =:
      c_{3}
      A_{3}(r)
      \,,
  \end{split}
\end{equation}
where
\begin{equation*}
  \begin{split}
    c_{3}
    &=
      c_{1}^{\frac{p-1}{p}}
      c_{2}^{\frac{1}{q}}
      \Big(
      \frac{\alpha_{1}+1}{\alpha_{1}}
      \Big)^{\frac{p+q(p-1)}{pq}}
    \\
    &=
      (p-1)^{\frac{p-1}{p}}
      \Big(
      \frac{\alpha_{2}(\alpha_{1}+1)}{\alpha_{1}(\alpha_{2}+1)}
      \Big)^{\frac{1}{q}}
      \alpha_{1}^{-\frac{1}{q}}
      \alpha_{2}^{-\frac{p-1}{p}}
      \,,
  \end{split}
\end{equation*}
and, by elementary calculations,
\begin{equation}
  \label{eq:aux_hardy_k}
  A_{3}(r)
  =
  F(r)
  \Big(
  \frac{\exw(r)}{r^{N}}
  \Big)^{\frac{1}{a}}
  \qquad
  F(r)
  =:
  \frac{r}{\exw(r)}
  e^{-\frac{ \exw(r)}{a}}
  \,.
\end{equation}
Note that \eqref{eq:powerlike} immediately implies the first inequality in
\begin{equation}
  \label{eq:aux_hardy_jk}
  \Big(
  \frac{\exw(r)}{r^{N}}
  \Big)^{\frac{1}{a}}
  \le
  \Big(
  \frac{\exw(r_{0})}{r_{0}^{N}}
  \Big)^{\frac{1}{a}}
  \le
  1
  +
  \frac{\exw(r_{0})^{\frac{1}{N}}}{r_{0}}
  \,,
  \qquad
  r
  \ge
  r_{0}
  \,,
\end{equation}
while the second inequality here follows from $a>N$, for all possible values of $\exw(r_{0})/r_{0}^{N}$.
\\
Then we are left with the task of estimating $F(r)$.
To this end we perform the change of variable
\begin{equation}
  \label{eq:aux_hardy_trsf}
  r
  =
  \exw^{(-1)}(
  a\lambda
  )
  \,,
  \qquad
  \lambda>0
  \,,
\end{equation}
mapping monotonically and surjectively $r\in(0,+\infty)$ to
$\lambda\in(0,+\infty)$.
We obtain that
\begin{equation}
  \label{eq:aux_hardy_kk}
  F(r(\lambda))
  =
  \frac{\exw^{(-1)}(a\lambda)}{a\lambda}
  e^{-\lambda}
  \,.
\end{equation}
We are going first to estimate $F(r(\lambda))$ for $\lambda>1$; thus we may appeal to \eqref{eq:aux_powerlike_inv} to get
\begin{equation}
  \label{eq:aux_hardy_kkk}
    F(r(\lambda))
    \le
      \frac{\exw^{(-1)}(a)}{a}
      \lambda^{\frac{1}{\alpha_{1}}-1}
      e^{-\lambda}
    \le
      c_{4}(\alpha_{1})
      \frac{\exw^{(-1)}(a)}{a}
      \,,
\end{equation}
where for all $0<\alpha<1$ we write
\begin{equation*}
  c_{4}(\alpha)
  =
  \max_{\lambda>0}
  \lambda^{\frac{1}{\alpha}-1}
  e^{-\lambda}
  =
  \Big(
  \frac{1}{\alpha}
  -
  1
  \Big)^{\frac{1}{\alpha}-1}
  e^{-(\frac{1}{\alpha}-1)}
  \,;
\end{equation*}
note that the bound in \eqref{eq:aux_hardy_kkk} is correct even if $1/\alpha_{1}-1\le1$. 
Next, by the same token, we bound for $0<\lambda\le 1$
\begin{equation}
  \label{eq:aux_hardy_kkkj}
    F(r(\lambda))
    \le
      \frac{\exw^{(-1)}(a)}{a}
      \lambda^{\frac{1}{\alpha_{2}}-1}
      e^{-\lambda}
    \le
      c_{4}(\alpha_{2})
      \frac{\exw^{(-1)}(a)}{a}
      \,.
\end{equation}
Finally, on collecting \eqref{eq:aux_hardy_ik}, as well as \eqref{eq:aux_hardy_jjj}, \eqref{eq:aux_hardy_jk}, \eqref{eq:aux_hardy_kkk}, \eqref{eq:aux_hardy_kkkj}, we infer \eqref{eq:aux_hardy_m} with $\CH\ge C$ as in \eqref{eq:aux_hardy_m}.
\end{proof}

\begin{lemma}
  \label{l:aux_bdd}
  Assume that $v\in W_{0}^{1,p}(B_{R})$, $R<+\infty$, that $\alpha_{2}\le 1$ and \eqref{eq:aux_emf_nn}.
  \\
  Then for $q$ as in \eqref{eq:aux_hardy_nn} and $a=pq/(q-p)$,
  \begin{equation}
    \label{eq:aux_bdd_n}
    \Big(
    \int_{B_{R}}
    \abs{v}^{q}
    \di\ew
    \Big)^{\frac{1}{q}}
    \le
    C
    \Big(
    \int_{B_{R}}
    \abs{\grad v}^{p}
    \di\ew
    \Big)^{\frac{1}{p}}
    \fkf(R)^{\frac{N}{a}-1}
    \,.
  \end{equation}
  Here $C=C(N,p,\alpha_{1},\alpha_{2})$.
\end{lemma}

\begin{proof}
  We start from the standard Sobolev inequality for $w\in W_{0}^{1,p}(B_{R})$
  \begin{equation}
    \label{eq:aux_bdd_i}
    \Big(
    \int_{B_{R}}
    \abs{w}^{p^{*}}
    \di x
    \Big)^{\frac{1}{p^{*}}}
    \le
    C(N,p)
    \Big(
    \int_{B_{R}}
    \abs{\grad w}^{p}
    \di x
    \Big)^{\frac{1}{p}}
    \,,
    \quad
    p^{*}
    =
    \frac{Np}{N-p}
    \,.
  \end{equation}
  We take in \eqref{eq:aux_bdd_i} $w=ve^{\exw/p}$. Note that owing to \eqref{eq:aux_non_non} and \eqref{eq:powerlike} we have $\exw'(s)\le \alpha_{2}\exw(1)s^{\alpha_{1}-1}$ for $0<s<1$, so that $\abs{\grad \exw}\in L^{p}_{\textup{loc}}(\RN)$, since $p<N$. Then by means of a direct calculation, and by using \eqref{eq:aux_emf_prpos} and Lemma~\ref{l:aux_poinc} we get for $C=C(N,p,\alpha_{1},\alpha_{2})$
  \begin{equation}
    \label{eq:aux_bdd_ii}
    \begin{split}
      \int_{B_{R}}
      \abs{v}^{p^{*}}
      e^{\frac{p^{*}}{p}\exw}
      \di x
      &\le
      C
        \Big(
        \int_{B_{R}}
        \big(
        \abs{\grad v}^{p}
        +
        \abs{\grad\exw}^{p}
        \abs{v}^{p}
        \big)
        e^{\exw}
        \di x
        \Big)^{\frac{p^{*}}{p}}
        \\
      &\le
        C
        \Big(
        \int_{B_{R}}
        \big(
        \abs{\grad v}^{p}
        +
        \fkf(\abs{x})^{p}
        \abs{v}^{p}
        \big)
        e^{\exw}
        \di x
        \Big)^{\frac{p^{*}}{p}}
        \\
      &\le
        C
        \Big(
        \int_{B_{R}}
        \abs{\grad v}^{p}
        e^{\exw}
        \di x
        \Big)^{\frac{p^{*}}{p}}
        \,.        
    \end{split}
  \end{equation}
  Let $q\in(p,p^{*})$. On applying the H\"older inequality, and recalling $e^{\exw p^{*}/q}\le e^{\exw p^{*}/p}$, $e^{\exw p/q}\le e^{\exw}$ we have
  \begin{equation}
    \label{eq:aux_bdd_iii}
    \begin{split}
      \int_{B_{R}}
      \abs{v}^{q}
      e^{\exw}
      \di x
      &\le
        \Big(
        \int_{B_{R}}
        \abs{v}^{p^{*}}
        e^{\frac{p^{*}}{p}\exw}
        \di x
        \Big)^{\frac{q-p}{p^{*}-p}}
        \Big(
        \int_{B_{R}}
        \abs{v}^{p}
        e^{\exw}
        \di x
        \Big)^{\frac{p^{*}-q}{p^{*}-p}}
        \\
      &\le
        C
        \Big(
        \int_{B_{R}}
        \abs{\grad v}^{p}
        e^{\exw}
        \di x
        \Big)^{\frac{p^{*}}{p}\,\frac{q-p}{p^{*}-p}}
        \Big(
        \int_{B_{R}}
        \abs{v}^{p}
        e^{\exw}
        \di x
        \Big)^{\frac{p^{*}-q}{p^{*}-p}}
        \,,
    \end{split}
  \end{equation}
  where in the last inequality we applied \eqref{eq:aux_bdd_ii}. Note that, as $q<p^{*}$, the $C$ in \eqref{eq:aux_bdd_iii} may be taken independent of $q$.
  \\
  Next we take into account that for $\alpha_{2}\le 1$ the function $s\mapsto \exw(s)/s$ is non-increasing by \eqref{eq:aux_non_non}, so that for $s<R$ we have by \eqref{eq:aux_emf_prpos}
  \begin{equation}
    \label{eq:aux_bdd_iv}
    \fkf(s)
    \ge
    \frac{\alpha_{1}}{\alpha_{1}+1}
    \frac{\exw(s)}{s}
    \ge
    \frac{\alpha_{1}}{\alpha_{1}+1}
    \frac{\exw(R)}{R}
    \ge
    \frac{\alpha_{1}}{\alpha_{1}+1}
    \frac{\alpha_{2}+1}{\alpha_{2}}
    \fkf(R)
    \,.
  \end{equation}
  Then we may use again Lemma~\ref{l:aux_poinc}  to bound the last integral in \eqref{eq:aux_bdd_iii} by
  \begin{equation}
    \label{eq:aux_bdd_v}
    \begin{split}
      \int_{B_{R}}
      \abs{v}^{p}
      e^{\exw}
      \di x
      &\le
        C(p,\alpha_{1},\alpha_{2})
        \fkf(R)^{-p}
        \int_{B_{R}}
        \abs{v}^{p}
        \fkf(\abs{x})^{p}
        e^{\exw}
        \di x
        \\
      &\le
        C
        \fkf(R)^{-p}
        \int_{B_{R}}
        \abs{\grad v}^{p}
        e^{\exw}
        \di x
        \,.
    \end{split}
  \end{equation}
  Finally we collect \eqref{eq:aux_bdd_iii} and \eqref{eq:aux_bdd_v} to infer
  \begin{equation}
    \label{eq:aux_bdd_vi}
    \begin{split}
      \int_{B_{R}}
      \abs{v}^{q}
      e^{\exw}
      \di x
      &\le
        C
        \fkf(R)^{-p\,\frac{p^{*}-q}{p^{*}-p}}
        \Big(
        \int_{B_{R}}
        \abs{\grad v}^{p}
        e^{\exw}
        \di x
        \Big)^{\frac{p^{*}}{p}\,\frac{q-p}{p^{*}-p}+ \frac{p^{*}-q}{p^{*}-p}}
      \\
      &=
      C
      \fkf(R)^{-q+\frac{Nq}{a}}
      \Big(
      \int_{B_{R}}
      \abs{\grad v}^{p}
      e^{\exw}
      \di x
      \Big)^{\frac{q}{p}}
      \,,
    \end{split}
  \end{equation}
  that is \eqref{eq:aux_bdd_n}.
\end{proof}

\section{Proof of Theorem~\ref{t:sur}}
\label{s:sur}

The proof is complex and we divide it in several steps. Preliminarily, we remark that the assumptions \eqref{eq:aux_emf_nn} and \eqref{eq:aux_emf_p} are in force, owing to the requirement $\alpha_{1}\ge \alpha_{2}/(\alpha_{2}+1)$ in \eqref{eq:sur_n}, and to the fact that $N\ge 2>2\alpha_{2}$.
\\
We also remark that by conservation of mass
\begin{equation}
  \label{eq:conservation}
  \int_{\RN}
  u(x,t)
  f(x)
  \di x
  =
  \int_{\RN}
  u_{0}(x)
  f(x)
  \di x
  =
  \norma{u_{0}\ew}{1}
  \,,
  \qquad
  t>0
  \,.
\end{equation}
We assume in this section  that $\norma{u_{0}\ew}{1}=1$, which is always possible by the remark that
\begin{equation*}
  U(x,t)
  =
  \lambda
  u(x,\lambda^{p+m-3}t)
\end{equation*}
is a solution to \eqref{eq:pde}, for any fixed $\lambda>0$.

\textbf{First step.}
As a first step, we obtain the $\sup$ estimate in \eqref{eq:sur_kk} below.
\\
We begin by stating the following inequality of Caccioppoli type: for any $a_{1}>a_{2}>0$, $\tau_{1}>\tau_{2}>0$, $s>1$ (and $s>3-m$ if $m<1$),
\begin{equation}
  \label{eq:sur_cacc}
  \begin{split}
    &\sup_{\tau_{1}<\tau<t}
      \int_{\RN}
      \ppos{u(\tau)-a_{1}}^{s}
      \di\ew
      +
      \int_{\tau_{1}}^{t}
      \int_{\RN}
      \abs{\grad\ppos{u-a_{1}}^{\theta}}^{p}
      \di\ew\di \tau
    \\
    &\quad
      \le
      c
      \Big(
      \frac{a_{1}}{a_{1}-a_{2}}
      \Big)^{\abs{m-1}}
      (\tau_{1}-\tau_{2})^{-1}
      \int_{\tau_{2}}^{t}
      \int_{\RN}
      \ppos{u-a_{2}}^{s}
      \di\ew
      \di \tau
      \,,
  \end{split}
\end{equation}
where $\theta=(p+m+s-3)/p$ and $c$ here and in the rest of this section denotes a constant depending on $m$, $p$, $N$, $s$, $\exw(1)$. Note that $s$ is fixed here for the rest of the proof. The proof of \eqref{eq:sur_cacc} is standard and we omit it; see also \cite[Lemma~2.17]{Andreucci:Tedeev:2021b}. Define $h_{0}>h_{\infty}>0$, $\tau_{0}>\tau_{\infty}>0$ and for $i\ge0$
\begin{equation*}
  k_{i}
  =
  h_{\infty}
  +
  (h_{0}-h_{\infty})
  2^{-i}
  \,,
  \quad
  t_{i}
  =
  \tau_{\infty}
  +
  (\tau_{0}-\tau_{\infty})
  2^{-i}
  \,,
  \quad
  v_{i}
  =
  \ppos{u-k_{i}}^{\theta}
  \,.
\end{equation*}
We take in \eqref{eq:sur_cacc} $a_{1}=k_{i}$, $a_{2}=k_{i+1}$, $\tau_{1}=t_{i}$, $\tau_{2}=t_{i+1}$ to get
\begin{equation}
  \label{eq:sur_i}
  \sup_{t_{i}<\tau<t}
  \int_{\RN}
  v_{i}^{d}
  \di\ew
  +
  \int_{t_{i}}^{t}
  \int_{\RN}
  \abs{\grad v_{i}}^{p}
  \di\ew\di \tau
  \le
  b^{i}\Psi
  \int_{t_{i+1}}^{t}
  \int_{\RN}
  v_{i+1}^{d}
  \di\ew
  \di \tau
  \,,
\end{equation}
where $b=2^{\abs{m-1}+1}$, $d=s/\theta<p$ and
\begin{equation*}
  \Psi
  :=
  c
  \Big(
  \frac{h_{0}}{h_{0}-h_{\infty}}
  \Big)^{\abs{m-1}}
  (\tau_{0}-\tau_{\infty})^{-1}
  \,.
\end{equation*}
We also set
\begin{equation*}
  \msw(\tau,k)
  =
  \int_{\{u(\tau)>k\}}
  \di\ew
  \,,
  \qquad
  k>0
  \,.
\end{equation*}
\\
On the other hand, by means of Lemma~\ref{l:aux_hardy} we have for $q\in(p,p^{*})$
\begin{equation}
  \label{eq:sur_hardy_i}
  \Big(
  \int_{\RN}
  v_{i+1}^{q}
  \di\ew
  \Big)^{\frac{1}{q}}
  \le
  \CH
  \Big(
  \int_{\RN}
  \abs{\grad v_{i+1}}^{p}
  \di\ew
  \Big)^{\frac{1}{p}}
  \,,
\end{equation}
where $\CH$ is the constant in \eqref{eq:aux_hardy_mm}; the number $q$ will be chosen below. Hence, on applying H\"older and Young inequalities we obtain from \eqref{eq:sur_hardy_i}
\begin{equation}
  \label{eq:sur_ii}
  \begin{split}
    &\int_{\RN}
    v_{i}(\tau)^{d}
    \di\ew
    \le
      \CH^{d}
      \Big(
      \int_{\RN}
      \abs{\grad v_{i+1}(\tau)}^{p}
      \di\ew
      \Big)^{\frac{d}{p}}
      \msw(\tau,k_{i+1})^{1-\frac{d}{q}}
      \\
    &\quad\le
      \frac{d}{p}
      \eps^{\frac{p}{d}}
      \int_{\RN}
      \abs{\grad v_{i+1}(\tau)}^{p}
      \di\ew
      +
      \frac{p-d}{p}
      \eps^{-\frac{p}{p-d}}
      \CH^{\frac{pd}{p-d}}
      \msw(\tau,k_{i+1})^{1+\frac{d(q-p)}{q(p-d)}}
      \,.
  \end{split}
\end{equation}
On integrating in time over $(t_{i+1},t)$ the inequality \eqref{eq:sur_ii} we get
\begin{equation}
  \label{eq:sur_iii}
  \begin{split}
    &b^{i}
      \Psi
      \int_{t_{i+1}}^{t}
      \int_{\RN}
      v_{i+1}^{d}
      \di\ew
      \di \tau
      \le
      \frac{d}{p}b^{i}
      \Psi
      \eps^{\frac{p}{d}}
      \int_{t_{i+1}}^{t}
      \int_{\RN}
      \abs{\grad v_{i+1}}^{p}
      \di\ew
      \di\tau
    \\
    &\quad+
      \frac{p-d}{p}b^{i}
      \Psi
      \eps^{-\frac{p}{p-d}}
      \CH^{\frac{pd}{p-d}}
      t
      \sup_{\tau_{\infty}<\tau<t}
      \msw(\tau,h_{\infty})^{1+\frac{d(q-p)}{q(p-d)}}
      \,,
  \end{split}
\end{equation}
for $b$ and $\Psi$ as in \eqref{eq:sur_i}.
Here we select $\eps$ as follows, for an $\eps_{1}>0$ to be fixed presently:
\begin{equation*}
  \frac{d}{p}b^{i}
  \Psi
  \eps^{\frac{p}{d}}
  =
  \eps_{1}
  \,.
\end{equation*}
Collecting \eqref{eq:sur_i} and \eqref{eq:sur_iii} we arrive at
\begin{equation}
  \label{eq:sur_iv}
  \begin{split}
    &\sup_{t_{i}<\tau<t}
      \int_{\RN}
      v_{i}^{d}
      \di\ew
      +
      \int_{t_{i}}^{t}
      \int_{\RN}
      \abs{\grad v_{i}}^{p}
      \di\ew\di \tau
      \le
      \eps_{1}
      \int_{t_{i+1}}^{t}
      \int_{\RN}
      \abs{\grad v_{i+1}}^{p}
      \di\ew
      \di\tau
    \\
    &\quad
      +
      b^{i\frac{p}{p-d}}
      \Psi^{\frac{p}{p-d}}
      \eps_{1}^{-\frac{d}{p-d}}
      \CH^{\frac{pd}{p-d}}
      t
      \sup_{\tau_{\infty}<\tau<t}
      \msw(\tau,h_{\infty})^{1+\frac{d(q-p)}{q(p-d)}}
      \,;
  \end{split}
\end{equation}
we omitted here from the last term the multiplicative constant $(p-d)d^{d/(p-d)}/p^{p/(p-d)}<1$.
\\
On iterating \eqref{eq:sur_iv} on $i$ (see e.g., \cite[(3.6)--(3.9)]{Andreucci:Tedeev:2021b}) we find for small enough $\eps_{1}$, e.g., $\eps_{1}=1/(2b^{p/(p-d)})$,
\begin{equation}
  \label{eq:sur_j}
  \begin{split}
    &\sup_{\tau_{0}<\tau<t}
      \int_{\RN}
      \ppos{u(\tau)-h_{0}}^{s}
      \di\ew
      \le
      c
      \Big(
      \frac{h_{0}}{h_{0}-h_{\infty}}
      \Big)^{\abs{m-1}\frac{p}{p-d}}
    \\
    &\quad\times
      (\tau_{0}-\tau_{\infty})^{-\frac{p}{p-d}}
      \CH^{\frac{pd}{p-d}}
      t
      \sup_{\tau_{\infty}<\tau<t}
      \msw(\tau,h_{\infty})^{1+\frac{d(q-p)}{q(p-d)}}
      \,.
  \end{split}
\end{equation}
We are going to apply \eqref{eq:sur_j} with $\tau_{0}=t_{n+1}'$, $\tau_{\infty}=t_{n}'$, $h_{0}=\bar\ell_{n}$, $h_{\infty}=\ell_{n}$, where for a $k>0$ to be chosen
\begin{equation*}
  \ell_{n}
  =
  k(1-2^{-n-1})
  \,,
  \quad
  \bar\ell_{n}
  =
  (\ell_{n}+\ell_{n+1})/2
  \,,
  \quad
  t_{n}'
  =
  t(1-2^{n-1})
  \,,
  \quad
  n\ge 0
  \,.
\end{equation*}
However we first note that by Chebyshev inequality
\begin{equation}
  \label{eq:sur_jj}
  \begin{split}
    Y_{n+1}
    &:=
    \sup_{t_{n+1}'<\tau<t}
    \msw(\tau,\ell_{n+1})
    \le
      \frac{1}{(\ell_{n+1}-\bar\ell_{n})^{s}}
      \sup_{t_{n+1}'<\tau<t}
      \int_{\RN}
      \ppos{u(\tau)-\bar\ell_{n}}^{s}
      \di\ew
      \\
    &=
      2^{(n+3)s}
      k^{-s}
      \sup_{t_{n+1}'<\tau<t}
      \int_{\RN}
      \ppos{u(\tau)-\bar\ell_{n}}^{s}
      \di\ew
      \,.
  \end{split}
\end{equation}
Then from \eqref{eq:sur_j}, \eqref{eq:sur_jj} we obtain
\begin{equation}
  \label{eq:sur_jjj}
    Y_{n+1}
    \le
      c
      b_{1}^{n}
      \CH^{\frac{pd}{p-d}}
      t^{-\frac{d}{p-d}}
      k^{-s}
      Y_{n}^{1+\frac{d(q-p)}{q(p-d)}}
      \,,
\end{equation}
where $b_{1}=2^{s+(1+\abs{m-1})p/(p-d)}$. Then from the iteration result \cite[Lemma~5.6, Chapter~II]{LSU} we infer that $Y_{n}\to0$ as $n\to+\infty$, provided
\begin{equation}
  \label{eq:sur_jv}
  k^{-1}
  t^{-\frac{1}{p+m-3}}
  \CH^{\frac{p}{p+m-3}}
  Y_{0}^{\frac{q-p}{q(p+m-3)}}
  \le
  c_{0}(N,m,p,s)
  \,,
\end{equation}
for a suitable $c_{0}>0$. Note that this implies $\norma{u(t)}{\infty}\le k$.
\\
But, again by Chebyshev inequality for any $r>1$
\begin{equation*}
  Y_{0}
  =
  \sup_{t/2<\tau<t}
  \msw(\tau,k/2)
  \le
  \frac{4^{r}}{k^{r}}
  \sup_{t/2<\tau<t}
  \int_{\RN}
  u(\tau)^{r}
  \di \ew
  \,.
\end{equation*}
The number $r$, as well as $q$, will be chosen below.
Next we select $k$ from
\begin{equation}
  \label{eq:sur_k}
  k^{-1}
  t^{-\frac{1}{p+m-3}}
  \CH^{\frac{p}{p+m-3}}
  \Big[
  \frac{4^{r}}{k^{r}}
  \sup_{t/2<\tau<t}
  \int_{\RN}
  u(\tau)^{r}
  \di \ew
  \Big]^{\frac{q-p}{q(p+m-3)}}
  =
  \frac{c_{0}}{2}
  \,.
\end{equation}
Then we get from some elementary algebra
\begin{equation}
  \label{eq:sur_kk}
    \norma{u(t)}{\infty}
    \le
      c
      t^{-\frac{q}{\Har}}
      \CH^{\frac{pq}{\Har}}
      \Big[
      4^{r}
      \sup_{t/2<\tau<t}
      \int_{\RN}
      u(\tau)^{r}
      \di \ew
      \Big]^{\frac{q-p}{\Har}}
    \,,
\end{equation}
where $\Har=q(p+m-3)+r(q-p)$. 
We have used that $q(p+m-3)/\Har\le 1$ to bound the constant $c$.

\textbf{Second step.}
As a second step in the proof of the Theorem, we find a suitably sharp integral estimate of $u^{r}$.
\\
To this end we introduce the notation
\begin{equation*}
  v
  =
  u^{\frac{r}{\lambda}}
  \,,
  \quad
  p>
  \lambda
  :=
  \frac{pr}{p+m+r-3}
  >
  \eta
  :=
  \frac{p}{p+m+r-3}
  \,;
\end{equation*}
note that $v^{\lambda}=u^{r}$, $v^{\eta}=u$.
Then we have from the differential equation \eqref{eq:pde} for $u$
\begin{equation}
  \label{eq:sur_u}
  \der{}{t}
  \int_{\RN}
  v^{\lambda}
  \di\ew
  =
  -
  r(r-1)
  \eta^{p}
  \int_{\RN}
  \abs{\grad v}^{p}
  \di\ew
  \,.
\end{equation}
On applying in turn H\"older inequality, conservation of mass and our assumption $\norma{u_{0}\ew}{1}=1$, and finally Lemma~\ref{l:aux_hardy}, we obtain
\begin{equation}
  \label{eq:sur_uu}
  \begin{split}
  \int_{\RN}
  v^{\lambda}
  \di\ew
  &\le
  \Big(
  \int_{\RN}
  v^{q}
  \di\ew
  \Big)^{\frac{\lambda-\eta}{q-\eta}}
  \Big(
  \int_{\RN}
  v^{\eta}
  \di\ew
  \Big)^{\frac{q-\lambda}{q-\eta}}
  \\
    &\le
      \CH^{q\frac{\lambda-\eta}{q-\eta}}
    \Big(
      \int_{\RN}
      \abs{\grad v}^{p}
      \di\ew
      \Big)^{\frac{q}{p}\,\frac{\lambda-\eta}{q-\eta}}
      \,.
  \end{split}
\end{equation}
From \eqref{eq:sur_u} and \eqref{eq:sur_uu} we get
\begin{equation}
  \label{eq:sur_uuu}
  \der{}{t}
  \int_{\RN}
  v^{\lambda}
  \di\ew
  \le
  -
  r(r-1)
  \eta^{p}
  \CH^{-p}
  \Big(
  \int_{\RN}
  v^{\lambda}
  \di\ew
  \Big)^{1+\omega}
  \,,
\end{equation}
where
\begin{equation*}
  1+\omega
  =
  \frac{p}{q}\,\frac{q-\eta}{\lambda-\eta}
  \,,
  \quad
  \text{i.e.,}
  \quad
  \omega
  =
  \frac{q(p+m-3)+q-p}{q(r-1)}
  >0
  \,.
\end{equation*}
Then, on integrating \eqref{eq:sur_uuu} on $(t/2,t)$ we obtain
\begin{equation}
  \label{eq:sur_x}
  \int_{\RN}
  v(t)^{\lambda}
  \di\ew
  \le
  \Big(
  \frac{2\CH^{p}}{\omega   r(r-1)
  \eta^{p}}
  \Big)^{\frac{1}{\omega}}
  t^{-\frac{1}{\omega}}
  \,.
\end{equation}
On applying \eqref{eq:sur_x} in \eqref{eq:sur_kk} we conclude
\begin{equation*}
  \norma{u(t)}{\infty}
  \le
  c
  t^{-\frac{q}{\Har}}
  \CH^{\frac{pq}{\Har}}
  \Big[
  4^{r}
  \Big(
  \frac{4\CH^{p}}{\omega   r(r-1)
    \eta^{p}}
  \Big)^{\frac{1}{\omega}}
  t^{-\frac{1}{\omega}}
  \Big]^{\frac{q-p}{\Har}}
  \,,
\end{equation*}
amounting to, via some algebra,
\begin{equation}
  \label{eq:sur_xx}
  \norma{u(t)}{\infty}
  \le
  c
  4^{\frac{r(q-p)}{\Har}}
  \Big[
  \frac{4}{\omega   r(r-1)
    \eta^{p}}
  \Big]^{\frac{q-p}{\omega\Har}}
  \CH^{\frac{pq}{q(p+m-3)+q-p}}
  t^{-\frac{q}{q(p+m-3)+q-p}}
  \,.
\end{equation}

\textbf{Third step.}
In the third and last step of the proof we select suitably the numbers $q$ and $r$, actually according to the constraint
\begin{equation}
  \label{eq:sur_y}
  r(q-p)
  =
  1
  \,.
\end{equation}
Note that $q$ belongs to the admissible range $(p,p^{*})$, since $r>1$ by assumption. In fact we are going to let $r\to+\infty$. On denoting
\begin{equation*}
  \sigma
  :=
  \frac{q}{p}
  \,,
  \quad
  \text{so that}
  \quad
  a
  =
  \frac{pq}{q-p}
  =
  \frac{q}{\sigma -1}
  \,,
\end{equation*}
as $r\to+\infty$ we have $q\to p+$ and $\sigma\to 1+$, $a\to+\infty$.
\\
Next we obtain a meaningful bound for $\CH$ in this limit;
according to the definitions above, we have $K(q,p)\le c$, and, also on invoking \eqref{eq:aux_powerlike_inv}, $\exw^{(-1)}(a)/a\to +\infty$. Thus for large $r$ we may take $r_{0}=1$ in \eqref{eq:aux_hardy_mm} and estimate
\begin{equation}
  \label{eq:sur_yy}
  \CH
  \le
  c
  \frac{\exw^{(-1)}(a)}{a}
  =
  c
  \exw^{(-1)}\Big(
  \frac{q}{\sigma-1}
  \Big)
  \frac{\sigma-1}{q}
  \le
  c
  \exw^{(-1)}\Big(
  \frac{1}{\sigma-1}
  \Big)
  (\sigma-1)
  \,,
\end{equation}
where in last inequality we used \eqref{eq:aux_powerlike_inv} again. We also remark that $\Har\to p(p+m-3)+1=:H$ and
\begin{gather}
  \label{eq:sur_yyy}
  \frac{r(q-p)}{\Har}
  \le
  \frac{1}{H}
  <
  1
  \,,
  \\
  \label{eq:sur_yv}  
  \frac{q-p}{\omega\Har}
  \le
  \frac{1}{p+m-3}
  \,,
  \\
  \label{eq:sur_vy}  
  \omega
  r(r-1)
  \eta^{p}
  \ge
  r^{-p+1}
  \frac{p^{p}(p+m-3)}{(p+m-2)^{p}}
  \,.
\end{gather}
Next we deal with the exponents of $t$ and of $\CH$ in \eqref{eq:sur_xx}; the exponent of $t$ can be bounded as in
\begin{equation}
  \label{eq:sur_vyy}
  \begin{split}
    &-
    \frac{q}{q(p+m-3)+q-p}
    \\
    &\quad=
    -
    \frac{1}{p+m-3}
    +
    \frac{q-p}{(p+m-3)[q(p+m-3)+q-p]}
    \\
    &\quad\le
    -
    \frac{1}{p+m-3}
    +
    \frac{q-p}{p(p+m-3)^{2}}
    \,.
  \end{split}
\end{equation}
In turn the exponent of $\CH$ can be majorized as in
\begin{equation}
  \label{eq:sur_vyyy}
  \frac{pq}{q(p+m-3)+q-p}
  \le
  \frac{p}{p+m-3}
  \,.
\end{equation}
Thus, on gathering \eqref{eq:sur_xx}--\eqref{eq:sur_vyyy} we obtain
\begin{equation}
  \label{eq:sur_z}
  \norma{u(t)}{\infty}
  \le
  c
  \Big[
  r^{p-1}
  \exw^{(-1)}\Big(
  \frac{1}{\sigma-1}
  \Big)^{p}
  (\sigma-1)^{p}
  \Big]^{\frac{1}{p+m-3}}
  t^{-\frac{1}{p+m-3}}
  t^{\frac{q-p}{p(p+m-3)^{2}}}
  \,.
\end{equation}
Finally we select for $t>e$, according to the ``logarithmic trick'' of \cite{Grillo:Muratori:2016},
\begin{equation*}
  r=\log t
  \,,
  \quad
  \text{so that}
  \quad
  \sigma
  =
  \frac{q}{p}
  =
  1
  +
  \frac{1}{p\log t}
  \,,
  \quad
  t^{q-p}
  =
  e
  \,,
\end{equation*}
and assuming that $t$ is so large that \eqref{eq:sur_yy} is valid,
\eqref{eq:sur_z} yields
\begin{equation}
  \label{eq:sur_zz}
  \begin{split}
    \norma{u(t)}{\infty}
    &\le
      c
      \Big[
      (\log t)^{p-1}
      \frac{
      \exw^{(-1)}(
      p\log t
      )^{p}
      }{
      (p\log t)^{p}
      }
      \Big]^{\frac{1}{p+m-3}}
      t^{-\frac{1}{p+m-3}}
      \\
    &\le
      c
      \Big[
      \frac{
      \exw^{(-1)}(
      \log t
      )^{p}
      }{
      \log t
      }
      \Big]^{\frac{1}{p+m-3}}
      t^{-\frac{1}{p+m-3}}
      \,,
  \end{split}
\end{equation}
where we used again \eqref{eq:aux_powerlike_inv}. The sought after estimate \eqref{eq:sur_nn} is proved.

\section{Proof of Theorem~\ref{t:bds}}
\label{s:bds}

We begin by remarking that \eqref{eq:aux_bdd_n} formally is the same as \eqref{eq:aux_hardy_m}, if we let $\CH=C\fkf(R)^{-1+N/a}$, excepting the fact that the latter applies to radial functions and the former to functions supported in $B_{R}$. Since in this proof we deal with a solution whose support (up to time $t$) is contained in $B_{R}$, we may repeat the proof of Theorem~\ref{t:sur} given in Section~\ref{s:sur} replacing \eqref{eq:aux_hardy_m} with \eqref{eq:aux_bdd_n} in each instance of its use. In this proof we take $R=c\exw^{(-1)}(\log t)$ for large $t$, according to \eqref{eq:fsp_n}; note that we still assume here $\norma{u_{0}\ew}{1}=1$. Here $c$ denotes a constant depending on $N$,  $p$, $m$, $\alpha_{1}$, $\alpha_{2}$, $\exw(1)$.

Therefore, with the formal replacement indicated above, we may immediately rewrite \eqref{eq:sur_xx} as
\begin{equation}
  \label{eq:bds_xx}
  \norma{u(t)}{\infty}
  \le
  c
  4^{\frac{r(q-p)}{\Har}}
  \Big[
  \frac{4}{\omega   r(r-1)
    \eta^{p}}
  \Big]^{\frac{q-p}{\omega\Har}}
  \fkf(R)^{-\frac{pq-N(q-p)}{q(p+m-3)+q-p}}
  t^{-\frac{q}{q(p+m-3)+q-p}}
  \,.
\end{equation}
Next we reason for $R$ large enough, i.e., for $t$ large enough, to imply $\fkf(R)<1$. Then we may bound in \eqref{eq:bds_xx}
\begin{equation}
  \label{eq:bds_i}
  \fkf(R)^{-\frac{pq-N(q-p)}{q(p+m-3)+q-p}}
  \le
  c
  \fkf(\exw^{(-1)}(\log t))^{-\frac{p}{p+m-3}}
  \,.
\end{equation}
On appealing again to \eqref{eq:sur_yyy}--\eqref{eq:sur_vyy} we get in this way, on selecting $q$ and $r$ by \eqref{eq:sur_y} and $r=\log t$ 
\begin{equation}
  \label{eq:bds_j}
  \begin{split}
    \norma{u(t)}{\infty}
    &\le
    c
    r^{\frac{p-1}{p+m-3}}
    \fkf(\exw^{(-1)}(\log t))^{-\frac{p}{p+m-3}}
    t^{-\frac{1}{p+m-3}}
    t^{\frac{q-p}{p(p+m-3)^{2}}}
    \\
    &\le
    c
    (\log t)^{\frac{p-1}{p+m-3}}
    \fkf(\exw^{(-1)}(\log t))^{-\frac{p}{p+m-3}}
    t^{-\frac{1}{p+m-3}}
    \,.
  \end{split}
\end{equation}
Finally note that by means of \eqref{eq:aux_emf_prpos} we may bound
\begin{equation}
  \label{eq:bds_jj}
  \fkf(\exw^{(-1)}(\log t))^{-1}
  =
  \excf'(\exw^{(-1)}(\log t))^{-1}
  \le
  \frac{\alpha_{1}+1}{\alpha_{1}}
  \frac{\exw^{(-1)}(\log t)}{\log t}
  \,.
\end{equation}
On applying \eqref{eq:bds_jj} in \eqref{eq:bds_j} we immediately infer \eqref{eq:sur_nn}.

\section{Proof of Theorem~\ref{t:fsp}}
\label{s:fsp}

For given $\eta$, $\sigma\in(0,1/4]$, $R\ge 4R_{0}$, we set
\begin{gather*}
  R_{n}'
  =
  \frac{R}{2}
  (1-\eta-\sigma+\sigma 2^{-n})
  \,,
  \quad
  R_{n}''
  =
  R
  (1+\eta+\sigma-\sigma 2^{-n})
  \,,
  \\
  A_{n}
  =
  \{
  x\in\RN
  \mid
  R_{n}'<\abs{x}<R_{n}''
  \}
  \subset
  A_{n+1}
  \,,
  \quad
  n\ge 0
  \,,
  \\
  A_{\infty}
  =
  \Big\{
  x\in\RN
  \mid
  \frac{R}{2}
  (1-\eta-\sigma)
  <\abs{x}<
  R
  (1+\eta+\sigma)
  \Big\}
  \,.
\end{gather*}
We also consider a sequence of cutoff functions $\zeta_{n}$ satisfying
\begin{gather*}
  0\le \zeta_{n}\le 1
  \,;
  \qquad
  \abs{\grad \zeta_{n}}\le\frac{2^{n+3}}{\sigma R}
  \,;
  \\
  \zeta_{n}(x)=1
  \,,
  \quad
  x\in A_{n}
  \,;
  \qquad
  \zeta_{n}(x)=0
  \,,
  \quad
  x\not\in A_{n+1}
  \,.
\end{gather*}
We have the following routine energy-type inequality
\begin{equation}
  \label{eq:fsp_i}
  \begin{split}
    J_{n}
    &:=
      \sup_{0<\tau<t}
      \int_{\RN}
      v_{n}^{d}
      \di\ew
      +
      \int_{0}^{t}
      \int_{\RN}
      \abs{\grad v_{n}}^{p}
      \di\ew
      \di\tau
    \\
    &\le
      c
      \frac{2^{np}}{\sigma^{p}R^{p}}
      \int_{0}^{t}
      \int_{\RN}
      v_{n+1}^{p}
      \di\ew
      \di\tau
      =:
      L_{n+1}
      \,,
  \end{split}
\end{equation}
where $s>1$, $d=s/\theta$, $\theta=(p+m+s-3)/p$, $v_{n}=(u\zeta_{n})^{\theta}$ and $c$ denotes a generic constant depending on $N$, $p$, $m$, $\alpha_{1}$, $\alpha_{2}$, $s$. In \eqref{eq:fsp_i} of course we used the fact that $\supp u_{0}\cap A_{n}=\emptyset$. The number $s(p,m,\alpha_{2})>1$ will be chosen below close enough to $1$. 

We recall the standard Gagliardo-Nirenberg inequality, valid for example for smooth compactly supported functions,
\begin{equation}
  \label{eq:fsp_ii}
  \int_{\RN}
  \abs{w}^{p}
  \di x
  \le
  c
  \Big(
  \int_{\RN}
  \abs{\grad w}^{p}
  \di x
  \Big)^{\xi}
  \Big(
  \int_{\RN}
  \abs{w}^{\eps}
  \di x
  \Big)^{\frac{p(1-\xi)}{\eps}}
  \,,
\end{equation}
where $\eps\in(0,p)$ will be chosen below as $\eps(p,m,s)$ and $\xi\in(0,1)$ is defined by
\begin{equation*}
  \frac{N}{p}
  =
  \frac{\xi(N-p)}{p}
  +
  \frac{N(1-\xi)}{\eps}
  \,,
  \quad
  \text{i.e.,}
  \quad
  \xi
  =
  \frac{N(p-\eps)}{N(p-\eps)+p\eps}
  \,.
\end{equation*}
We apply \eqref{eq:fsp_ii} with $w=v_{n+1}\ew^{1/p}$, obtaining
\begin{equation}
  \label{eq:fsp_iii}
    \int_{\RN}
    \abs{v_{n+1}}^{p}
    \ew
    \di x
    \le
    c
    \Big(
    \int_{\RN}
    \abs{\grad (v_{n+1}\ew^{\frac{1}{p}})}^{p}
    \di x
    \Big)^{\xi}
    \Big(
    \int_{\RN}
    \abs{v_{n+1}}^{\eps}
      \ew^{\frac{\eps}{p}}
    \di x
    \Big)^{\frac{p(1-\xi)}{\eps}}
    \,.
\end{equation}
Note that $\abs{\grad g(\abs{x})}\le c \fkf(\abs{x})$ by \eqref{eq:aux_emf_prpos}, so that
\begin{equation}
  \label{eq:fsp_iv}
  \begin{split}
    \int_{\RN}
    \abs{\grad (v_{n+1}\ew^{\frac{1}{p}})}^{p}
    \di x
    &\le
      c
      \int_{\RN}
      (
      \abs{\grad v_{n+1}}^{p}
      \ew
      +
      v_{n+1}^{p}
      \abs{\grad \exw}^{p}
      \ew
      )
      \di x
    \\
    &\le
      c
      \int_{\RN}
      (
      \abs{\grad v_{n+1}}^{p}
      \ew
      +
      v_{n+1}^{p}
      \fkf(\abs{x})^{p}
      \ew
      )
      \di x
    \\
    &\le
      c
      \int_{\RN}
      \abs{\grad v_{n+1}}^{p}
      \ew
      \di x
      \,,
  \end{split}
\end{equation}
where in last inequality we applied \eqref{eq:aux_poinc_n}. Then from \eqref{eq:fsp_iii}, \eqref{eq:fsp_iv} we infer
\begin{equation}
  \label{eq:fsp_vvv}
  \begin{split}
  \int_{\RN}
    \abs{v_{n+1}}^{p}
    \ew
    \di x
    &\le
    c
    \Big(
    \int_{\RN}
    \abs{\grad v_{n+1}}^{p}
      \ew
    \di x
    \Big)^{\xi}
    \\
    &\quad\times
      \Big(
    \int_{\RN}
    \abs{v_{n+1}}^{\eps}
      \ew
    \di x
    \Big)^{\frac{p(1-\xi)}{\eps}}
      \ew\Big(
      \frac{R}{4}
      \Big)^{-\frac{(p-\eps)(1-\xi)}{\eps}}
    \,.
  \end{split}
\end{equation}
We have exploited the fact that $\supp v_{n+1}\subset \{\abs{x}\ge R/4\}$. Here we select
\begin{equation*}
  \eps
  =
  \frac{1}{\theta}
  <
  p
  \,,
  \quad
  \text{so that}
  \quad
  \xi
  =
  \frac{N(p+m+s-4)}{N(p+m+s-4)+p}
  \,.
\end{equation*}
Next we apply Young inequality in \eqref{eq:fsp_vvv}, and integrate in time, obtaining for $\delta>0$
\begin{equation}
  \label{eq:fsp_vi}
  \begin{split}
    L_{n+1}
    &\le
      \delta
      \int_{0}^{t}
      \int_{\RN}
      \abs{\grad v_{n+1}}^{p}
      \ew
      \di x
      \di\tau
      \\
    &\quad+
      b^{n}
      \frac{
      c
      \delta^{-\frac{1}{p}(N+p+m+s-4)}
      t
      }{
      (\sigma R)^{\Kas}
      \ew(R/4)^{p+m+s-4}
      }
      \Big(
      \sup_{0<\tau<t}
      \int_{A_{\infty}}
      u
      \ew
      \di x
      \Big)^{p+m+s-3}
      \,,
  \end{split}
\end{equation}
where $b=2^{p/(1-\xi)}$, $\Kas=N(p+m+s-4)+p$. From \eqref{eq:fsp_i} and \eqref{eq:fsp_vi} we get
\begin{equation}
  \label{eq:fsp_vii}
  J_{n}
  \le
  \delta
  J_{n+1}
  +
  b^{n}
  \frac{
    c
    \delta^{-\frac{1}{p}(N+p+m+s-4)}
    t
  }{
    (\sigma R)^{\Kas}
    \ew(R/4)^{p+m+s-4}
  }
  \Big(
  \sup_{0<\tau<t}
  \int_{A_{\infty}}
  u
  \ew
  \di x
  \Big)^{p+m+s-3}
  \,,
\end{equation}
to which we may apply the same iteration scheme which lead us from \eqref{eq:sur_iv} to \eqref{eq:sur_j}, arriving at
\begin{equation}
  \label{eq:fsp_j}
  \sup_{0<\tau<t}
  \int_{A_{0}}
  u^{s}
  \di\ew
  \le
  \frac{
    c
    t
  }{
    (\sigma R)^{\Kas}
    \ew(R/4)^{p+m+s-4}
  }
  \Big(
  \sup_{0<\tau<t}
  \int_{A_{\infty}}
  u
  \di\ew
  \Big)^{p+m+s-3}
  \,.
\end{equation}
Next we select $\eta=\sigma=2^{-i-2}$, $i\ge 0$, so that on defining
\begin{gather*}
  \bar R_{i}'
  =
  \frac{R}{2}
  (1- 2^{-i-2})
  \,,
  \quad
  \bar R_{i}''
  =
  R
  (1+ 2^{-i-2})
  \,,
  \\
  D_{i}
  =
  \{
  x\in\RN
  \mid
  \bar R_{i}'<\abs{x}<\bar R_{i}''
  \}
  \supset
  D_{i+1}
  \,,
  \quad
  i\ge 0
  \,,
  \\
  D_{\infty}
  =
  \Big\{
  x\in\RN
  \mid
  \frac{R}{2}
  <\abs{x}<
  R
  \Big\}
  \,,
\end{gather*}
the inequality \eqref{eq:fsp_j} yields
\begin{equation}
  \label{eq:fsp_jj}
  \sup_{0<\tau<t}
  \int_{D_{i+1}}
  u^{s}
  \di\ew
  \le
  \frac{
    c
    2^{i\Kas}
    t
  }{
    R^{\Kas}
    \ew(R/4)^{p+m+s-4}
  }
  Y_{i}^{p+m+s-3}
\end{equation}
where we also let
\begin{equation*}
  Y_{i}
  =
  \sup_{0<\tau<t}
  \int_{D_{i}}
  u
  \di\ew
  \,.
\end{equation*}
Note that
\begin{equation*}
  \int_{D_{i+1}}
  \di\ew
  \le
  c R^{N}
  \ew(2R)
  \,.
\end{equation*}
Then we apply H\"older inequality and \eqref{eq:fsp_jj}, to estimate
\begin{equation}
  \label{eq:fsp_jjj}
  \begin{split}
    Y_{i+1}
    &\le
      \Big(
      \sup_{0<\tau<t}
      \int_{D_{i+1}}
      u^{s}
      \di\ew
      \Big)^{\frac{1}{s}}
      \Big(
      \int_{D_{i+1}}
      \di\ew
      \Big)^{\frac{s-1}{s}}
    \\
    &\le
      c
      2^{\frac{\Kas}{s}i}
      \Big(
        \frac{
      t
      }{
      R^{\Ka}
      \ew(R/4)^{p+m-3}
      }
      Y_{i}^{p+m-3}
      \Big)^{\frac{1}{s}}
      \Big(
      \frac{\ew(2R)}{\ew(R/4)}
      \Big)^{\frac{s-1}{s}}
      Y_{i}
      \,,
  \end{split}
\end{equation}
where $\Ka=N(p+m-3)+p$.
\\
Thus, on invoking  \cite[Lemma~5.6, Chapter~II]{LSU} we have that $Y_{i}\to 0$ as $i\to+\infty$, provided
\begin{equation}
  \label{eq:fsp_jv}
  \frac{
    t
  }{
    R^{\Ka}
    \ew(R/4)^{p+m-3}
  }
  Y_{0}^{p+m-3}
  \Big(
  \frac{\ew(2R)}{\ew(R/4)}
  \Big)^{s-1}
  \le
  c_{0}
\end{equation}
for a suitable $c_{0}=c_{0}(N,p,m,s)<1$. This limiting relation of course implies that $u(t)=0$ in $D_{\infty}$, i.e., boundedness of support. In turn, \eqref{eq:fsp_jv} is implied, owing to conservation of mass, by
\begin{equation}
  \label{eq:fsp_v}
  \frac{
    t
  }{
    \ew(R/4)^{p+m-3}
  }
  \norma{u_{0}\ew}{1}^{p+m-3}
  \ew(2R)^{s-1}
  \le
  c_{0}
  \,,
\end{equation}
where we have dropped the non-essential factors $\ew(R/4)^{1-s}<1$ (due to the exponential character of $\ew$) and $R^{-\Ka}<1$ (due to the fact that we are going to assume $R\ge 1$). Then we remark that, owing to \eqref{eq:aux_non_non},
\begin{equation}
  \label{eq:fsp_vj}
  \begin{split}
    \frac{\ew(2R)^{s-1}}
    {\ew(R/4)^{p+m-3}}
    &=
      \exp\big(
      (s-1)\exw(2R)
      -
      (p+m-3)\exw(R/4)
      \big)
    \\
    &\le
      \exp\big(
      [(s-1)2^{\alpha_{2}}
      -
      (p+m-3)4^{-\alpha_{2}}
      ]
      \exw(R)
      \big)
      \,.
  \end{split}
\end{equation}
Thus we select for example
\begin{gather*}
  s
  =
  1
  +
  \frac{1}{2}
  8^{-\alpha_{2}}
  (p+m-3)
  >1
  \,,
  \\
  \text{and let}
  \quad
  \nu
  =
  -[
  (s-1)2^{\alpha_{2}}
  -
  (p+m-3)4^{-\alpha_{2}}
  ]
  >0
  \,.
\end{gather*}
Hence \eqref{eq:fsp_v} is implied by
\begin{equation}
  \label{eq:fsp_vjj}
  e^{-\nu \exw(R)}
  t
  \norma{u_{0}\ew}{1}^{p+m-3}
  \le
  c_{0}
  \,,
\end{equation}
that is by, for a suitable $c\ge \exw(1)$, $\exw(4R_{0})$ (so that $R\ge 1$, $4R_{0}$),
\begin{equation}
  \label{eq:fsp_vjjj}
  \begin{split}
    \exw(R)
    &\ge
    c
    \log(e+  t
    \norma{u_{0}\ew}{1}^{p+m-3}
    )
      \\
    &\ge
    \frac{1}{\nu}
    \log(  t
    \norma{u_{0}\ew}{1}^{p+m-3}
    )
    +
    \frac{1}{\nu}
    \log\frac{1}{c_{0}}
    \,,
  \end{split}
\end{equation}
which amounts to \eqref{eq:fsp_n}, after an application of \eqref{eq:aux_powerlike_inv}.



\def\cprime{$'$}

\end{document}